\documentclass[10pt]{amsart}
\usepackage[margin=1.2in]{geometry}

\usepackage{kotex}
\usepackage{colortbl}
\usepackage[pagewise]{lineno}

\newcommand{\ignore}[1]{}
\newtheorem{theorem}{Theorem}[section]
\newtheorem{remark}[theorem]{Remark}
\newtheorem{lemma}[theorem]{Lemma}
\newtheorem{proposition}[theorem]{Proposition}

\newtheorem{definition}[theorem]{Definition}

\numberwithin{equation}{section}

\newcommand{\bbr}{\mathbb{R}}

\newcommand{\<}{\langle}
\renewcommand{\>}{\rangle}

\usepackage{graphicx, subfigure}
\usepackage[export]{adjustbox}

\begin{document}

\title[Kinetic description for the herding behavior in financial market]{A kinetic description for the herding behavior in financial market}

\author{Hyeong-Ohk Bae}
\address{Department of Financial Engineering, Ajou University, Suwon, Korea (Republic of)}
\email{hobae@ajou.ac.kr}
\author{Seung-yeon Cho }
\address{Department of Mathematics, Sungkyunkwan University, Suwon 440-746, Korea (Republic of)}
\email{chosy89@skku.edu}
\author{Jeongho Kim}
\address{Department of Mathematical Sciences, Seoul National University, Seoul 08826,  Korea (Republic of)}
\email{jhkim206@snu.ac.kr}
\author{SEOK-BAE YUN }
\address{Department of Mathematics, Sungkyunkwan University, Suwon 440-746, Korea (Republic of)}
\email{sbyun01@skku.edu}





\begin{abstract}
As a continuation of the study of the herding model proposed in \cite{BCLY}, we consider in this paper the derivation of the kinetic version of the herding model, the existence of the measure-valued solution and the corresponding herding behavior at the kinetic level.
We first consider the mean-field limit of the particle herding model and derive the existence of the measure-valued solutions for the kinetic herding model.
We then study the herding phenomena of the solutions in two different ways by introducing two different types of herding energy functionals. First, we derive a herding phenomena of the measure-valued solutions under virtually no restrictions on the parameter sets using the Barbalat's lemma. We, however, don't get any herding rate in this case. On the other hand, we also establish a Gr\"onwall type estimate for another herding functional, leading to the exponential herding rate, under comparatively strict conditions. These results are then extended to smooth solutions.
\end{abstract}
\maketitle


\section{Introduction}
The collective behaviors of many body systems are easily found in nature and society, and numerous mathematical models \cite{BD,CKJRF,CS,DM,DCBC,MAL,VCBCS} have been introduced to understand such collective phenomena. Recently, a particle type herding model was introduced to describe herding behaviors in financial market \cite{BCLY}:

\begin{align}\label{A-1}
\begin{split}
\frac{dx_i}{dt}&=v_i,\quad t>0,\quad i=1,2,\ldots, N,\cr
\frac{dv_i}{dt}&
=\frac{\lambda_x}{N}\sum_{j=1}^N\phi\big(|x_j{-}x_i|\big)(x_j-x_i)
+\frac{\lambda_v}{N}\sum_{j=1}^N\phi\big(|x_j{-}x_i|\big)\left(v_j-v_i\right)\\
&+\lambda_w \big(w(x,t)-x_i\big).
\end{split}
\end{align}
Here,  $N$ and $d$ denote respectively the numbers of market players and assets   in the financial market. For $i=1,2,\ldots, N$, $x_i=(x_{i}^1,\ldots, x_{i}^d)\in\bbr^d$ represents $i$-th participant's subjective price or value of the assets $1,2,\ldots, d$, and $v_i=(v_{i}^1,\ldots, v_{i}^d)$ represents the favorability of the $i$-th participant to assets $1,2,\ldots,d$.
The function $\phi$ is the communication rate and $\lambda_x$, $\lambda_v$ and $\lambda_w$ are interaction strength. This model describes the herding behavior in the market in which a consensus emerges in the subject price and favorability.

The first equation in \eqref{A-1} says that the subject price is affected by the favorability on those assets. The first term in the right-hand side of the second equation implies that the favorability is affected by the other player's subjective prices, and the second term describes that the favorability is affected by other player's favorability. The third term describes the effect of market signal.

Throughout this paper, we assume that the communication rate $\phi:\bbr\to\bbr^+$ is a smooth Lipschitz continuous and satisfies following conditions:
\begin{align}\label{Phi condition}
\phi(r)=\phi(-r)\quad \mbox{and} \quad (\phi(r)-\phi(s))(r-s)\le 0
\end{align}
for $r,s\ge0$.

The market signal $w(x,t)$ can take various forms according to market situations (see \cite{BCLY}). In this paper, we focus on the case where the ensemble average of the subjective prices plays the role of market signal affecting the expectation and favorability of market players, in which case, $w(x,t)$ is given by
\[w(x,t)=\frac{1}{N}\sum_{i=1}^N x_i =:x_c.\]

It was shown in \cite{BCLY} that herding phenomena occurs in the following sense:
\[\lim_{t\to \infty} |x_{i}-x_j|=\lim_{t\to\infty} |v_i-v_j|=0\]
for a very general initial conditions.

Before we move on to the kinetic model, a few more words on what led to the re-interpretation of the subjective price of the assets and the favorability as the position and the velocity in the phase space are in order.
Starting from the geometric Brownian motion of a stock price,
we derive a relation of stock price $S(t)$ and rate of return $\mu$ given by $\frac{d}{dt} \log \mathbb{E}(S) = \mu$, where $\mathbb{E}[\cdot]$ denotes the expectation.
Introducing new variables $x$ and $v$ by $(x,v)=(\log \mathbb{E}(S), \mu)$ and interpreting them as the player's subjective value and favorability respectively, we arrive at the relation: $\frac{dx}{dt}=v$, which enables one to reinterpret the two market variables $x$ and $v$ as the position and velocity variables in the phase space. For further details, we refer to \cite[Section 2]{BCLY}.

 A standard BBGKY argument (see Section 2) shows that the kinetic mean-field model corresponding to the particle model \eqref{A-1} is given by

\begin{align}\label{A-2}
\frac{\partial{f}}{\partial{t}}+v \cdot \nabla_{x}f +\nabla_{v} \cdot \big(Q(f)f\big)=0,
\end{align}
where the non-local herding force $Q(f)$ is defined by
\begin{equation}\label{A-3}
Q(f)=\lambda_w Q_{1}(f) + \lambda_x Q_2(f) + \lambda_v Q_3(f),
\end{equation}
and
\begin{align}
\begin{aligned}\label{A-4}
&Q_1(f)= \int_{\mathbb{R}^{2d}}   (x^*-x) f(x^*,v^*,t) dx^* dv^*,\\
&Q_2(f)= \int_{\mathbb{R}^{2d}}\phi(|x-x^*|)(x^*-x) f(x^*,v^*,t) dx^* dv^*,\\
&Q_3(f)= \int_{\mathbb{R}^{2d}}\phi(|x-x^*|)(v^*-v) f(x^*,v^*,t) dx^* dv^*.
\end{aligned}
\end{align}

The distribution function $f(x,v,t)$ denotes the number density of market players with given subjective price $x \in \mathbb{R}^d$ and favorability $v \in \mathbb{R}^d$ at time $t \ge 0 $ .

\subsection{Main result}
The main goal of this paper is two-folded. First, we rigorously verify the mean-field limit  from the microscopic model \eqref{A-1} toward mesoscopic model \eqref{A-2}. In order to state the mean-field limit result, we start with the definition of the measure-valued solutions. Let $\mathcal{M}(\bbr^{2d})$ be a space of nonnegative Radon measure, which is a dual of $C_0(\bbr^{2d})$. For all $\nu\in\mathcal{M}(\bbr^{2d})$ and $g\in C_0(\bbr^{2d})$, we denote their duality as
\[\<\nu,g\>:=\int_{\bbr^{2d}}g(x,v)\nu(dx,dv).\]
Using this notation, we define the measure-valued solution of \eqref{A-2} as follows.

\begin{definition}\label{D3.1}
For $T\in(0,\infty)$, $\mu\in L^\infty([0,T);\mathcal{M}(\bbr^{2d}))$ is a measure-valued solution to (\ref{A-2}) with initial measure $\mu_0$ if and only if the following conditions hold:
\begin{enumerate}
\item $\mu$ is weakly continuous:
\[\<\mu_t,g\> \mbox{ is continuous function with respect to time $t$ for all $g\in C_0(\bbr^{2d})$.}\]
\item $\mu$ satisfies the following integral equation for all $g\in C_0^1(\bbr^{2d}\times [0,T))$:
\[\<\mu_t,g(\cdot,\cdot,t)\>-\<\mu_0,g(\cdot,\cdot,0)\>=\int_0^t \<\mu_s,\partial_s g+v\cdot \nabla_xg+Q\cdot\nabla_vg\>\,ds,\]
where $Q(x,v,\mu_s)$ is a non-local forcing term given by
\begin{align*}
&Q(x,v,\mu_s):=\int_{\bbr^{2d}}\bigg[\lambda_w(x^*-x)+ \lambda_x\phi(|x-x^*|)(x^*-x)+ \lambda_v\phi(|x-x^*|)(v^*-v)\bigg]\mu_s(dx^*,dv^*).
\end{align*}

\end{enumerate}
\end{definition}

 Now, we present our first result on the existence and the uniqueness of measure-valued solutions of equation \eqref{A-2}.
\begin{theorem}\label{T.1.1}
Let $\mu_0$ be a Radon measure with a compact support. Assume the communication rate satisfies (\ref{Phi condition}) and
\[
0\le\phi(r)\le \phi_M,
\]
for some positive constant $\phi_M$.
Then, there exists a unique measure-valued solution $\mu\in L^\infty([0,T);\mathcal{M}(\bbr^{2d}))$ to \eqref{A-2} with initial data $\mu_0$. Moreover, let $\mu_t$ and $\nu_t$ be solutions to \eqref{A-2} with initial data $\mu_0$ and $\nu_0$ respectively. Then there exists a positive constant $C=C_T$ such that the following stability estimate holds:
\[
W_1(\mu_t,\nu_t)\le C_TW_1(\mu_0,\nu_0),
\]
where $W_1(\cdot,\cdot)$ denotes the Wasserstein-1 distance (See Definition \ref{Wasser}).
\end{theorem}

Now, we present our results on the herding phenomena of the measure-valued solutions.
Both theorems state that the variation of $x$ and $v$ around the mean subjective price $x_c(t)$ and mean favorability $v_c(t)$ defined as
\[
x_c(t) = \int_{\bbr^{2d}}x\,\mu_t(dx,dv), \quad v_c(t)= \int_{\bbr^{2d}}v\,\mu_t(dx,dv)
\]
vanishes asymptotically.
The first herding estimate (Theorem \ref{T.1.2}) presents the decay of the variations of $x$ and $v$ over a general condition on parameters and communication rate. No decay estimate, however, is available in this case.

\begin{theorem}\label{T.1.2}
Let $\mu_t$ be the measure-valued solution to \eqref{A-2} subjected to initial data $\mu_0$. Assume the communication rate satisfies (\ref{Phi condition})
and
\begin{align*}
		\phi_m \leq \phi(r) \leq \phi_M,
\end{align*}
for some positive constants $\phi_m$ and $\phi_M$.
		Then, $\mu_t$ shows the herding behavior in the following sense:
\[
\lim_{t\to\infty}\int_{\bbr^{2d}}(|x-x_c(t)|^2+|v-v_c(t)|^2)\mu_t(dx,dv)=0.
\]
\end{theorem}

Our second herding estimate  provides the exponential decay of variation. The price to pay is the restriction on the parameters.
\begin{theorem}\label{T.1.3}
		Let $\mu_t$ be the measure-valued solution to \eqref{A-2} subjected to initial data $\mu_0$.  Assume the communication rate satisfies (\ref{Phi condition})
and
\begin{align*}
		\phi_m \leq \phi \leq \phi_M,
\end{align*}
for some positive constants $\phi_m$ and $\phi_M$. We further assume that $\lambda_x$, $\lambda_v$ and $\lambda_w$ satisfy
		\begin{align}\label{C-12}
		\frac{ \phi_M\lambda_x^2}{\phi_m\lambda_w \lambda_v^2} <\frac{1}{2}\min(1,\phi_M).
		\end{align}
		Then, there exist positive constants $C$ and  $\beta$ such that
		\[\int_{\bbr^{2d}}(|x-x_c(t)|^2+|v-v_c(t)|^2)\mu_t(dx,dv)\le C e^{-\frac{\beta}{2} t},\]
		where $C\equiv C(\mu_0, \lambda_w, \lambda_x, \lambda_v)$ is a positive constant and $\beta\equiv \beta(\phi_m,\phi_M, \lambda_w, \lambda_x, \lambda_v)$ is defined in \eqref{def beta}.
\end{theorem}


The proof of the existence part mainly follows, up to additional complications, the argument in \cite{HL}. The differentiation from the standard Cucker-Smale type argument arises in the analysis of the asymptotic behavior.
In the Cucker-Smale kinetic equation, the velocity deviation alone satisfies a Gr\"onwall inequality, yielding flocking estimates (See \cite{HL,HT}). In our kinetic herding model, however, due to the presence of $Q_1$ and $Q_2$, the interaction between the $x$ and $v$ variables plays more important role than in the case of the Cucker-Smale model, and as such, the time derivatives of the deviation functionals of $x$ and $v$ give rise to several covariance type terms.

To take care of this issue, we introduce two herding energy functionals $E(t)$ and $K(t)$ by combining the deviation functionals of $x$ and $v$ with an appropriately chosen
potential functional or a covariance functional, that turn out to be non-negative and non-increasing under appropriate conditions on the parameter set.

The decay for $E(t)$ is derived under virtually no restrictions on the parameter set using the Barbalat's lemma. The decay rate, however, is not available in this case. On the other hand, we can close a Gr\"onwall type estimate for another herding functional $K(t)$ to derive an exponential herding rate. Instead, the admissible parameter set is far more restricted in the second case.

\subsection{Literature review}

The collective behavior and emergence of organized motion arising from non-coordinated interactions are getting much attention recently since they are frequently observed in nature and society. Most of the mathematical models suggested to describe such collective behaviors fall into one of the three categories:
particle models, kinetic models and continuum models.

 Particle models consider a group of self-propelled particles following simple rules, such as adjusting the velocity or frequency in comparison to other particles, or keeping  minimal distances from others, to name a few. Typical examples of modeling of swarms of self-propelled particles in nature are the flocking of birds or milling of fish and herding of sheeps. Vicek et al \cite{VCBCS} proposed a discrete model to study the emergence of self-organized behavior by the alignment of directions of moving objects with a constant speed.
In \cite{CS}, the authors introduced so-called the Cucker-Smale model to describe the flocking behavior arising from the adjustment of velocities with respect to the
velocities of other agents in the group. Precise conditions for the flocking and sharp flocking rates of the Cucker-Smale model was derived in \cite{CFRT,HL,HT}.
In \cite{MT}, authors adjusted the Cucker-Smale model in such a way that the clustering configuration of agents can be reflected in the communication rate.
Milling behavior emergent from the combining effect of self-propelling, friction and  attractive-repulsive potential was studied in \cite{CDP,DCBC}.

Such self-propelled particle description also has been successfully applied to describe the collective dynamics in human society.
In \cite{KHKJ}, the emergence of cultural classes is described by the assimilation and distinction between agents, based on the first order C-S model.
The flocking phenomena of stochastic volatility is studied in \cite{ABHKL13,BHKLLY15}. 

In \cite{BCLY},  a particle-herding model \eqref{A-1} is proposed to investigate the investors' herding behaviors in financial
markets, where the terminology `herding' is used to describe individual's tendency to follow others
regardless of one's own opinion. The modelling of the movement of pedestrians in the frame work of self-propelled particle can be found in \cite{DAMPT,HM}.

The kinetic description is also popular for modelling collective dynamics of a large number of agents at the mesoscopic level. For this, a number density function $f$ over the phase space $(x, v)$ is introduced, and the kinetic equation governing the time evolution of $f$ is derived. 
Roughly, one can divide kinetic models into two cases. The first case is the Vlasov type equations obtained by suitable  mean-field limits of aforementioned particle models.
The mean field limit of the aforementioned discrete models in \cite{CKJRF,VCBCS} are derived in \cite{DM}. The mean-field limit of Cucker-Smale models are considered in \cite{CFRT,HL,HT}. We also refer to kinetic models regarding self-propelling particles with attraction and repulsion effects \cite{CCR,CFTV}, and the roosting force \cite{CKMT}.

On the other hand, there are kinetic models that does not arise from the mean-field limit, but from a direct modelling assumptions.
	In \cite{DJT}, herding behaviors of agents in a market is described by an inhomogeneous Boltzmann type kinetic equation using two variables, namely, the agent's estimated asset value and irrationality. In \cite{CPT}, the wealth distribution in the market is modeled based on binary exchanges of their money. The formation of individual's opinion is also modelled in \cite{T} using  opinion variables.
The crowd dynamics in an unbounded domain is dealt with in \cite{BBK}, and generalized to bounded domain case \cite{BG}. There are also several works on
the traffic flows \cite{NS,PH}. For the works on social-behavioral dynamics, see  \cite{BDG} for the modeling of behavioral learning dynamics, and \cite{MBG} for social dynamics.

We keep the reference check on the continuum approach to minimum, since they are out of the scope of this work.
Regarding the hydrodynamic limit of aforementioned kinetic models toward corresponding macroscopic models, we refer to \cite{CCR,DM,FK,HT,KMT}. We  mention the work on the dynamics of interacting particles in fluids \cite{BCHK, BCHK2}.

Our literature review is far from exhaustive since there are extensive literature out there. We refer interested readers to nice surveys and lectures in \cite{BD,CFTV,CHL,DLO,NPT,PT} for further references.\newline\newline



\noindent \textbf{Notations.} Throughout the paper, we will use $|\cdot |$ as the standard $\ell_2$ norm in $\bbr^d$.  The constant $C$ stands for generic constant, which can be different from line to line. We also use $C_{T}$ when it is necessary to show the dependence clearly.

\section{Derivation of Kinetic Equation}

In this section, we first present the formal derivation of mean-field Vlasov-type equation \eqref{A-2} from its particle equation \eqref{A-1} using the formal BBGKY hierarchy for the case of $w(x,t)=x_c(t)$. Rigorous justification in the framework of measure-valued solutions will be given in the next section. First of all, let $f_N:=f_N(x_{1},v_{1},\dots,x_{N},v_{N},t)$ be a $N$-particle distribution satisfying the following Liouville equation:
\begin{align*}
\frac{\partial{f_N}}{\partial{t}}+\sum_{i=1}^{N}v_i \cdot \nabla_{x_i}f_N +\sum_{i=1}^{N} \nabla_{v_i} \cdot \left( Q(x,v)  f_N \right)=0,
\end{align*}
where
\begin{align*}
Q(x,v):&= \frac{\lambda_w}{N}\sum_{j=1}^N(x_j-x_i)\\
&+\frac{\lambda_x}{N}\sum_{j=1}^{N} \phi(|x_i-x_j|)(x_j-x_i)+\frac{\lambda_v}{N}\sum_{j=1}^{N}\phi(|x_i-x_j|)(v_j-v_i).
\end{align*}

In order to obtain a Vlasov-type equation, we take marginal for $N$-particle distribution function $f_N$ with respect to $(x_2,v_2),(x_3,v_3),\ldots,(x_N,v_N)$ to have
\begin{align*}
&\int_{\mathbb{R}^{2(N-1)d}}\frac{\partial{f_N}}{\partial{t}}dx_2 dv_2...dx_N dv_N = \frac{\partial{f_1}}{\partial{t}}(x_1,v_1,t),\cr
&\int_{\mathbb{R}^{2(N-1)d}}\sum_{i=1}^{N}v_i \cdot \nabla_{x_i}f_N dx_2 dv_2...dx_N dv_N = v_1 \cdot \nabla_{x_1}f_1(x_1,v_1,t),\cr
&\int_{\mathbb{R}^{2(N-1)d}}\sum_{i=1}^{N} \nabla_{v_i} \cdot \left(Q(x,v)  f_N \right) dx_2 dv_2...dx_N dv_N\cr
&\quad=\frac{\lambda_w}{N}\int_{\mathbb{R}^{2(N-1)d}}\sum_{i=1}^{N} \nabla_{v_i} \cdot \left(\sum_{j=1}^N(x_j-x_i)  f_N \right) dx_2 dv_2...dx_N dv_N\cr
&\quad+\frac{\lambda_x}{N}\int_{\mathbb{R}^{2(N-1)d}}\sum_{i=1}^{N} \nabla_{v_i} \cdot \left(\sum_{j=1}^N\phi(|x_j-x_i|)(x_j-x_i)  f_N \right) dx_2 dv_2...dx_N dv_N
\end{align*}
\begin{align*}
&\quad+\frac{\lambda_v}{N}\int_{\mathbb{R}^{2(N-1)d}}\sum_{i=1}^{N} \nabla_{v_i} \cdot \left(\sum_{j=1}^N\phi(|x_j-x_i|)(v_j-v_i)  f_N \right) dx_2 dv_2...dx_N dv_N\cr
&\quad=:\mathcal{I}_{11}+\mathcal{I}_{12}+\mathcal{I}_{13}.
\end{align*}
We now impose the assumption that $f_N$ is invariant under changing particle labels with the mean-field limit, $N \rightarrow \infty$:
\[f_1(x_1,v_1,t)=f_1(x_2,v_2,t)=\cdots=f_1(x_N,v_N,t),\]
\[f_2(x_1,v_1,x_2,v_2,t)=f_2(x_1,v_1,x_3,v_3,t)=\cdots=f_2(x_1,v_1,x_N,v_N,t),\]
and adopt the \textit{molecular chaos assumption}:
\[f_2(x_1,v_1,x_2,v_2,t)=f_1(x_1,v_1,t)f_1(x_2,v_2,t),\]
to compute the three terms in $\mathcal{I}_{1i}$, $i=1,2,3$ as
\begin{align*}
\mathcal{I}_{11}&= \frac{\lambda_w}{N}\nabla_{v_1} \cdot \int_{\mathbb{R}^{2(N-1)d}} \sum_{j=1}^{N}(x_j-x_1) f_N dx_2 dv_2...dx_N dv_N\\
&=\frac{N-1}{N}\lambda_w \nabla_{v_1} \cdot \int_{\mathbb{R}^{2d}}   (x_2-x_1) f_2 dx_2 dv_2\\
&=\frac{N-1}{N}\lambda_w \nabla_{v_1} \cdot \int_{\mathbb{R}^{2d}}   (x_2-x_1) f_1(x_1,v_1,t)f_1(x_2,v_2,t)\,dx_2 dv_2,
\end{align*}	
\begin{align*}
\mathcal{I}_{12}&=\frac{N-1}{N}\lambda_x\nabla_{v_1} \cdot \int_{\mathbb{R}^{2d}}\phi(x_1-x_2)(x_2-x_1) f_2 \,dx_2 dv_2\\
&=\frac{N-1}{N}\lambda_x\nabla_{v_1} \cdot \int_{\mathbb{R}^{2d}}\phi(x_1-x_2)(x_2-x_1) f_1(x_1,v_1,t)f_1(x_2,v_2,t)\,dx_2 dv_2
\end{align*}
and
\begin{align*}
\mathcal{I}_{13}&=\frac{N-1}{N}\lambda_v\nabla_{v_1} \cdot \int_{\mathbb{R}^{2d}}\phi(x_1-x_2)(v_2-v_1) f_2 \,dx_2 dv_2\\
&=\frac{N-1}{N}\lambda_v\nabla_{v_1} \cdot \int_{\mathbb{R}^{2d}}\phi(x_1-x_2)(v_2-v_1) f_1(x_1,v_1,t)f_1(x_2,v_2,t)\,dx_2 dv_2.\\
\end{align*}
We now take the limit $N \rightarrow \infty$ to obtain the limiting function $f(x,v,t):=\lim\limits_{N\rightarrow\infty} f_1(x_1,v_1,t)$ satisfying the following Vlasov-type equation
\[
\frac{\partial{f}}{\partial{t}}+v \cdot \nabla_{x}f +\nabla_{v} \cdot \big(Q(f)f\big)=0
\]
with the non-local operator $Q$ defined in \eqref{A-3}--\eqref{A-4}.  

\section{Mean-field limit and measure-valued solutions}
In the previous section, we figured out that \eqref{A-2} is the correct kinetic equation for the particle herding model \eqref{A-1} through the formal BBGKY argument.
With the knowledge of the exact form of the kinetic model for \eqref{A-1}, we prove in this section that the Vlasov type equation \eqref{A-2} really  is the kinetic limit of \eqref{A-1}, by showing that the empirical measure constructed from \eqref{A-1} converges in a Wasserstein metric to the distribution measure satisfying \eqref{A-2} in the sense of the measure-valued solution defined in Definition \ref{D3.1}.
For simplicity, we normalize the interaction strengths:  $\lambda_w=\lambda_x=\lambda_v=1$ throughout this section.

\subsection{Preliminary estimates}
In this subsection, we provide several \textit{a priori} estimates on the moments. We start with the conservation laws.
	\begin{lemma}\label{L3.2}
		Let $\mu_t\in L^\infty([0,T);\mathcal{M}(\bbr^{2d}))$ be a compact supported measure-valued solution of \eqref{A-2}. Then,
		\[\frac{d}{dt}\int_{\bbr^{2d}}\mu_t(dx,dv)=0,\quad \frac{d}{dt}\int_{\bbr^{2d}}x\mu_t(dx,dv)=\int_{\bbr^{2d}}v\mu_t(dx,dv),\quad \frac{d}{dt}\int_{\bbr^{2d}}v\mu_t(dx,dv)=0.\]
	\end{lemma}
	\begin{proof}
		It comes directly from the Definition \ref{D3.1} (2)  by taking $g(x,v,t):=1,x,v$.
	\end{proof}
	\begin{remark}\label{R3.1}
		From the Lemma \ref{L3.2}, we have
		\[\int_{\bbr^{2d}}\mu_t(dx,dv)=\int_{\bbr^{2d}}\mu_0(dx,dv)=: m_0,\quad \int_{\bbr^{2d}}v\mu_t(dx,dv)=\int_{\bbr^{2d}}v\mu_0(dx,dv)=:m_1.\]
	\end{remark}
	For brevity, we put $m_0=1$ and define three functionals $a,b,c$ as follows:
	\begin{align}
	\begin{aligned}\label{C-0}
	Q(x,v,\mu_s)&=\left[\int_{\bbr^{2d}} x^*\mu_s(dx^*,dv^*)\right]+\left[\int_{\bbr^{2d}}\phi(|x-x^*|)(x^*+v^*)\mu_s(dx^*,dv^*)\right]\\
	&-\left[1+\int_{\bbr^{2d}}\phi(|x-x^*|)\mu_s(dx^*,dv^*)\right]x-\left[\int_{\bbr^{2d}}\phi(|x-x^*|)\mu_s(dx^*,dv^*)\right]v\\
	&=:a(\mu_s)+b(x,\mu_s)-(1+c(x,\mu_s))x-c(x,\mu_s)v.
	\end{aligned}
	\end{align}

In the following lemma, we show that the second moment of measure-valued solution is bounded.
\begin{lemma}\label{L3.3}
Let $\mu_t\in L^\infty([0,T);\mathcal{M}(\bbr^{2d}))$ be a compact-supported measure-valued solution of \eqref{A-2}. Then,
\[\int_{\bbr^{2d}}(|x|^2+|v|^2)\mu_t(dx,dv)\le C_T.\]
\end{lemma}
\begin{proof}
We use the definition of measure-valued solution to obtain
\[\frac{d}{dt}\int_{\bbr^{2d}}|x|^2\mu_t(dx,dv)=2\int_{\bbr^{2d}}x\cdot v\mu_t(dx,dv)\]
and
\begin{align*}
&\frac{d}{dt}\int_{\bbr^{2d}}|v|^2\mu_t(dx,dv)\\
&=2\int_{\bbr^{2d}}v\cdot Q(x,v,\mu_t)\mu_t(dx,dv)=2\int_{\bbr^{2d}}v\cdot\Bigg[\int_{\bbr^{2d}}(x^*-x)\mu_t(dx^*,dv^*)\\
&+\int_{\bbr^{2d}}\phi(|x-x^*|)(x^*-x)\mu_t(dx^*,dv^*)+\int_{\bbr^{2d}}\phi(|x-x^*|)(v^*-v)\mu_t(dx^*,dv^*)\Bigg]\,\mu_t(dx,dv)\\
&=2\left(\int_{\bbr^{2d}}v\mu_t(dx,dv)\right)\left(\int_{\bbr^{2d}}x^*\mu_t(dx^*,dv^*)\right)-2\int_{\bbr^{2d}}x\cdot v\mu_t(dx,dv)\\
&+2\int_{\bbr^{4d}}\phi(|x-x^*|)v\cdot(x^*-x)\mu_t(dx,dv)\mu_t(dx^*,dv^*)\\
&-\int_{\bbr^{4d}}\phi(|x-x^*|)|v^*-v|^2\mu_t(dx,dv)\mu_t(dx^*,dv^*).
\end{align*}
Therefore, we have
\begin{align*}
\frac{d}{dt}&\int_{\bbr^{2d}}(|x|^2+|v|^2)\mu_t(dx,dv)\\
&\le 2|m_1|\int_{\bbr^{2d}}|x|\mu_t(dx,dv)+2\phi_M\left(\int_{\bbr^{2d}}|v|\mu_t(dx,dv)\right)\left(\int_{\bbr^{2d}}|x^*|\mu_t(dx^*,dv^*)\right)\\
&+2\phi_M\int_{\bbr^{2d}}|xv|\mu_t(dx,dv)\le C\int_{\bbr^{2d}}(|x|^2+|v|^2)\mu_t(dx,dv),
\end{align*}
where we use the Cauchy-Schwartz inequality in the last inequality. Hence, we conclude that
\begin{equation*}
\int_{\bbr^{2d}}(|x|^2+|v|^2)\mu_t(dx,dv)\le e^{CT}\int_{\bbr^{2d}}(|x|^2+|v|^2)\mu_0(dx,dv).
\end{equation*}
\end{proof}

\begin{lemma}\label{L3.4}
Let $a,b,c$ be the functionals defined in \eqref{C-0}.
 Then, the following estimates hold:
\[|a(\mu_t)| \le C_T,\quad |b(x,\mu_t)|\le 2\phi_MC_T,\quad |c(x,\mu_t)|\le \phi_M.\]

\end{lemma}
\begin{proof}
This is straightforward if one uses Cauchy-Schwartz inequality and Lemma \ref{L3.3}.
\end{proof}

Now, for $(x,v,t)\in \bbr^d\times\bbr^d\times[0,T)$ and $\mu\in L^\infty([0,T);\mathcal{M}(\bbr^{2d}))$, we define the bi-characteristic curve which passes $(x,v)$ at time $t$ by
\begin{align}
\begin{aligned}\label{C-1}
\frac{d}{ds}X_\mu(s;t,x,v)&=V_\mu(s;t,x,v),\quad 0<s\le T,\\
\frac{d}{ds}V_\mu(s;t,x,v)&=Q\big(X_\mu(s;t,x,v),V_\mu(s;t,x,v),\mu_s\big),
\end{aligned}
\end{align}
where $Q$ is defined in \eqref{C-0}.
Moreover, we define the radius of supports of the measure-valued solutions as follows:
\[R_x(t):=\sup \{|x|~|~\mbox{$\exists~v$ s.t.~$(x,v)\in~$supp$(\mu_t)$}\},\quad R_v(t):=\sup \{|v|~|~\mbox{$\exists~x$ s.t.~$(x,v)\in~$supp$(\mu_t)$}\}.\]

Now, we estimate the support of a measure-valued solution in the next lemma.
\begin{lemma}\label{L3.5}
Let $\mu_t\in L^\infty([0,T);\mathcal{M}(\bbr^{2d}))$ be a compact supported measure-valued solution of \eqref{A-2} which satisfies the uniform boundedness of moments:
\[\int_{\bbr^{2d}}\mu_t(dx,dv)= 1,\quad \int_{\bbr^{2d}}|v|^2\mu_t(dx,dv)\le m_2.\]
Then, the support of the measure-valued solution is uniformly bounded by constant depending on $t$:
\[R_x(s)<C(t),\quad R_v(s)<C(t),\quad 0\le s\le t.\]
\end{lemma}
\begin{proof}
Consider the characteristic curve for velocity \eqref{C-1}$_2$ which starts at $(x,v)$ on time $t=0$:
\begin{align*}
\frac{d}{ds}V_\mu(s;0,x,v)&=a(\mu_s)+b(X_\mu(s;0,x,v),\mu_s)\\
&-(1+c(X_\mu(s;0,x,v),\mu_s))X_\mu(s;0,x,v)-c(X_\mu(s;0,x,v),\mu_s)V_\mu(s;0,x,v).
\end{align*}
Therefore, we have
\begin{align*}
&\frac{d}{ds}|X_\mu(s;0,x,v)|\le |V_\mu(s;0,x,v)|,\\
&\frac{d}{ds}|V_\mu(s;0,x,v)|\le C_T+2\phi_MC_T+(1+\phi_M)|X_\mu(s;0,x,v)|+\phi_M|V_\mu(s;0,x,v)|.
\end{align*}
We add two inequalities to get
\[\frac{d}{ds}(|X_\mu(s;0,x,v)|+|V_\mu(s;0,x,v)|)\le C_T+C_T(|X_\mu(s;0,x,v)|+|V_\mu(s;0,x,v)|).\]
Then, the Gr\"onwall lemma gives
\[\sup_{0\le t\le T}\big(|X_\mu(t;0,x,v)|+|V_\mu(t;0,x,v)|\big)\le C_T,\]
which implies the boundedness of support of the measure-valued solution.
\end{proof}

\begin{lemma}\label{L3.6}
Let $\mu_t\in L^\infty([0,T);\mathcal{M}(\bbr^{2d}))$ be a measure-valued solution of \eqref{A-2}. Then for any test function $h\in C^1_0(\bbr^{2d})$,
\[\int_{\bbr^{2d}}h(x,v)\mu_t(dx,dv)=\int_{\bbr^{2d}}h(X_\mu(t;s,x,v),V_\mu(t;s,x,v))\mu_s(dx,dv).\]
\end{lemma}

\begin{proof}
For any $h\in C_0^1(\bbr^{2d})$, we define
\[g(x^*,v^*,\tau):=h(X_\mu(t;\tau,x^*,v^*),V_\mu(t;\tau,x^*,v^*)),\]
so that
\begin{equation}\label{C-2}
g(X_\mu(\tau;t,x,v),V_\mu(\tau;t,x,v),\tau)=h(x,v).
\end{equation}
We differentiate \eqref{C-2} with respect to $\tau$ to get
\[\partial_\tau g+v^*\cdot \nabla_{x}g+Q\cdot\nabla_{v}g=0,\]
since right-hand side of \eqref{C-2} is independent of $\tau$. Therefore, inserting this choice of $g$ into the identity
in Definition \ref{D3.1} (1), we obtain
\[\<\mu_t,g(\cdot,\cdot,t)\>=\<\mu_s,g(\cdot,\cdot,s)\>,\]
which implies
\begin{align*}
&\int_{\bbr^{2d}}h(x,v)\mu_t(dx,dv)\\
&\qquad=\int_{\bbr^{2d}}h(X_\mu(t;t,x,v),V_\mu(t;t,x,v))\mu_t(dx,dv)
=\int_{\bbr^{2d}}g(x,v,t)\mu_t(dx,dv)\\
&\qquad=\<\mu_t g(\cdot,\cdot,t)\>=\<\mu_s,g(\cdot,\cdot,s)\>=\int_{\bbr^{2d}}g(x,v,s)\mu_s(dx,dv)\\
&\qquad=\int_{\bbr^{2d}}h(X_\mu(t;s,x,v),V_\mu(t;s,x,v))\mu_s(dx,dv).
\end{align*}
\end{proof}

\subsection{Stability analysis}
In this subsection, we provide stability analysis of the measure-valued solution of \eqref{A-2} up to any finite time.
Estimates in this subsection will be used crucially to show the existence and
the uniqueness of the measure-valued solution of \eqref{A-2} in the next subsection.
 We start with the review on the definition of Wasserstein distance.
\begin{definition}\label{Wasser}
Let $\mu,\nu$ be two Radon measure on $\bbr^{d}$.
Then the Wasserstein-$p$ distance $W_p(\mu,\nu)$ is defined by
\[W_p(\mu,\nu):=\inf_{\gamma\in\Gamma(\mu,\nu)}\left(\int_{\bbr^{2d}}|z_1-z_2|^p\gamma(dz_1,dz_2)\right)^{1/p},\]
where $\Gamma(\mu,\nu)$ denotes the set of  all probability measures whose marginals are $\mu$ and $\nu$.
\end{definition}
 In particular, when $p=1$, the Wasserstein-1 distance has following equivalent expression:
\[W_1(\mu,\nu)=\sup_{g\in\Omega} \left|\int_{\bbr^{d}}g(z)\big(\mu(dz)-\nu(dz)\big)\right|,\]
where $\Omega$ is given by
\[\Omega:=\left\{g:\bbr^{d}\to \bbr~:~\|g\|_\infty:=\sup_{z\in\bbr^d}|g(z)| \le1,\quad \|g\|_{\textup{Lip}}:=\sup_{z_1\neq z_2\in\bbr^d}\frac{|g(z_1)-g(z_2)|}{|z_1-z_2|}\le 1\right\}.\]

\begin{remark}\label{R3.2}~
\begin{enumerate}
\item The space of all Radon measure with finite $p$-th moment $P_p(\bbr^d)$ with topology induced by the Wasserstein-$p$ distance is Polish space.
\item For all $g\in C_0(\bbr^{d})$, we have
\[\left|\int_{\bbr^{d}}g(z)\mu(dz)-\int_{\bbr^{d}}g(z)\nu(dz)\right|\le \max\{\|g\|_\infty,\|g\|_{\textup{Lip}}\}W_1(\mu,\nu).\]
\end{enumerate}
\end{remark}

Now, we provide several estimates using the Wasserstein-1 distance.

\begin{lemma}\label{L3.8}
Let $\mu_t,\nu_t\in L^\infty([0,T);\mathcal{M}(\bbr^{2d}))$ be two measure-valued solutions to \eqref{A-2}. Then, we have the following estimates:
\begin{align*}
&(i)~|a(\mu_t)-a(\nu_t)|\le \max\left\{1,C_T\right\}W_1(\mu_t,\nu_t),\\
&(ii)~|b(x,\mu_t)-b(x,\nu_t)|\le \max\left\{2C_T\phi_M,2C_T\|\phi\|_{\textup{Lip}}+\sqrt{2}\phi_M\right\}W_1(\mu_t,\nu_t),\\
&(iii)~|c(x,\mu_t)-c(x,\nu_t)|\le \max\left\{\|\phi\|_{\textup{Lip}},\phi_M\right\}W_1(\mu_t,\nu_t).
\end{align*}
\end{lemma}
\begin{proof}
We only prove $(ii)$. $(i)$ and $(iii)$ can be proved similarly.
Owing to Lemma \ref{L3.5}, it is clear that
\[|\phi(|x-x^*|)(x^*+v^*)|\le 2C_T\phi_M,\]
and
\begin{align*}
&\big|\phi(|x-x_1^*|)(x_1^*+v_1^*)-\phi(|x-x_2^*|)(x_2^*+v_2^*)\big|\\
&\qquad\le \big|(\phi(|x-x_1^*|)-\phi(|x-x_2^*|))(x_1^*+v_1^*)\big|
+\big|\phi(|x-x_2^*|)(x_1^*+v_1^*-x_2^*-v_2^*)\big|\\
&\qquad\le 2C_T\|\phi\|_{\textup{Lip}}|(x_1^*,v_1^*)-(x_2^*,v_2^*)|+\sqrt{2}\phi_M|(x_1^*,v_1^*)-(x_2^*,v_2^*)|\\
&\qquad=(2C_T\|\phi\|_{\textup{Lip}}+\sqrt{2}\phi_M)|(x_1^*,v_1^*)-(x_2^*,v_2^*)|.
\end{align*}

We use the definition of $b(x,\mu_t)$ and Remark \ref{R3.2} (2) to see
\begin{align*}
|b(x,\mu_t)-b(x,\nu_t)|&=\left|\int_{\bbr^{2d}}\phi(|x-x^*|)(x^*+v^*)\big(\mu(dx^*,dv^*)-\nu(dx^*,dv^*)\big)\right|\\
&\le \max\left\{2C_T\phi_M,2C_T\|\phi\|_{\textup{Lip}}+\sqrt{2}\phi_M\right\}W_1(\mu_t,\nu_t).
\end{align*}
\end{proof}

\begin{lemma}\label{L3.9}
Let $\mu_t,\nu_t\in L^\infty([0,T);\mathcal{M}(\bbr^{2d}))$ be two measure-valued solutions to \eqref{A-2}. Then, we have the following estimations:
\begin{align*}
	(i)~|b(X_\mu,\mu_t)-b(X_\nu,\nu_t)| &\le C\|\phi\|_{\textup{Lip}}|X_\mu(t)-X_\nu(t)|\\
	&+\max\{2C_T\phi_M,2C_T\|\phi\|_{\textup{Lip}}+\sqrt{2}\phi_M\}W_1(\mu_t,\nu_t).\\
	(ii)~|c(X_\mu,\mu_t)-c(X_\nu,\nu_t)|&\le\|\phi\|_{\textup{Lip}}|X_\mu(t)-X_\nu(t)|+\max\{\|\phi\|_{\textup{Lip}},\phi_M\}W_1(\mu_t,\nu_t).
\end{align*}
\end{lemma}

\begin{proof}
We focus on (i) since (ii) can be proved similarly.
 To estimate $|b(X_\mu,\mu_t)-b(X_\nu,\nu_t)|$, we separate it by two terms as follows:
\begin{align*}
&|b(X_\mu,\mu_t)-b(X_\nu,\nu_t)|\\
&\quad=\left|\int_{\bbr^{2d}}\phi(|X_\mu-x^*|)(x^*+v^*)\mu_t(dx^*,dv^*)-\int_{\bbr^{2d}}\phi(|X_\nu-x^*|)(x^*+v^*)\nu_t(dx^*,dv^*)\right|\\
&\quad\le \int_{\bbr^{2d}}\big|\phi(|X_\mu-x^*|)-\phi(|X_\nu-x^*|)\big| |x^*+v^*|\mu_t(dx^*,dv^*)\\
&\quad+\left|\int_{\bbr^{2d}}\phi(|X_\nu-x^*|)(x^*+v^*)
\big(\mu_t(dx^*,dv^*)
-\nu_t(dx^*,dv^*)\big)\right|\\
&\quad=:\mathcal{I}_{21}+\mathcal{I}_{22}.
\end{align*}
$\bullet$ (Estimation of $\mathcal{I}_{21}$) We use the Lipschitz continuity of $\phi$ and Lemma \ref{L3.3}, together with Cauchy-Schwartz inequality to get
\[\mathcal{I}_{21}\le C_T\|\phi\|_{\textup{Lip}}|X_\mu-X_\nu|.\]
$\bullet$ (Estimation of $\mathcal{I}_{22}$) Estimation of $\mathcal{I}_{22}$ directly comes from Lemma \ref{L3.8} (ii):
\[\mathcal{I}_{22} \le \max\left\{2C_T\phi_M,2C_T\|\phi\|_{\textup{Lip}}+\sqrt{2}\phi_M\right\}W_1(\mu_t,\nu_t).\]
Now, we combine estimation of $\mathcal{I}_{21}$ and $\mathcal{I}_{22}$ to get desired estimate.
\end{proof}

\begin{lemma}\label{lem3.10}
Let $\mu_t,\nu_t\in L^\infty([0,T);\mathcal{M}(\bbr^{2d}))$ be two measure-valued solutions of \eqref{A-2} with compactly supported initial condition $\mu_0$ and $\nu_0$. Moreover, suppose that $\mu_t$ and $\nu_t$ have uniformly bounded moments:
\begin{align*}
\int_{\bbr^{2d}}\mu_t(dx,dv)= 1,\quad \int_{\bbr^{2d}}|v|^2\mu_t(dx,dv)\le m_2,\\
 \int_{\bbr^{2d}}\nu_t(dx,dv)= 1,\quad \int_{\bbr^{2d}}|v|^2\nu_t(dx,dv)\le m_2.\end{align*}
Then, for $0\le s,t < T$, we have
\[
|X_\mu(s;t,x,v)-X_\nu(s;t,x,v)|+|V_\mu(s;t,x,v)-V_\nu(s;t,x,v)|
\le C_T\int_{\min\{s,t\}}^{\max\{s,t\}}W_1(\mu_\tau,\nu_\tau)\,d\tau.
\]
\end{lemma}
\begin{proof}
We define the differences of two characteristic curve as
\[\mathbb X(s):=X_\mu(s;t,x,v)-X_\nu(s;t,x,v),\quad \mathbb V(s):=V_\mu(s;t,x,v)-V_\nu(s;t,x,v).\]
Without loss of generality, we assume $0<t<s$. From \eqref{C-1}, it is clear that
\[\frac{d\mathbb X(\tau)}{d\tau}=\mathbb V(\tau),\quad \mathbb X(t)=0,\]
and hence,
\begin{equation}\label{C-3}
|\mathbb X(s)|\le \int_t^s|\mathbb V(\tau)|d\tau.
\end{equation}

On the other hand, again from \eqref{C-1} we can estimate $V(s)$ as
\begin{align}
\begin{aligned}\label{C-4}
\frac{d\mathbb V(\tau)}{d\tau}
&=a(\mu_\tau)-a(\nu_\tau)+b(X_\mu,\mu_\tau)-b(X_\nu,\nu_\tau)-\mathbb X(\tau)
-c(X_\mu,\mu_\tau)\mathbb X(\tau)\\
&-(c(X_\mu,\mu_\tau){-}c(X_\nu,\nu_\tau))X_\nu-c(X_\mu,\mu_\tau)\mathbb V(\tau)
-(c(X_\mu,\mu_\tau){-}c(X_\nu,\nu_\tau))V_\nu
\end{aligned}
\end{align}
with $\mathbb V(t)=0$.
Then, we multiply \eqref{C-4} by $e^{c(X_\mu,\mu_\tau)\tau}$, integrate over $[t,s]$, and use Lemma \ref{L3.8} to obtain
\begin{align}
\begin{aligned}\label{C-5}
|\mathbb V(s)|&\le \int_t^s |a(\mu_\tau)-a(\nu_\tau)|+|b(X_\mu,\mu_\tau)-b(X_\nu,\nu_\tau)|\,d\tau
+\int_t^s (1+c(X_\mu,\mu_\tau))|\mathbb X(\tau)|\,d\tau\\
&+\int_t^s|c(X_\mu,\mu_\tau)-c(X_\nu,\nu_\tau)|(|X_\nu|+|V_\nu|)\,d\tau\\
&\le \int_t^s \Big[\max\{1,C_T\}W_1(\mu_\tau,\nu_\tau)+C_T\|\phi\|_{\textup{Lip}}|\mathbb X(\tau)|\\
&+\max\left\{2C_T\phi_M,2C_T\|\phi\|_{\textup{Lip}}+\sqrt{2}\phi_M\right\}W_1(\mu_\tau,\nu_\tau)\Big]d\tau\\
&+\int_t^s(1+\phi_M)|\mathbb X(\tau)|\,d\tau+\int_t^s2C_T(\|\phi\|_{\textup{Lip}}|\mathbb X(\tau)|+\max\left\{\|\phi\|_{\textup{Lip}},\phi_M\right\}W_1(\mu_\tau,\nu_\tau))\,d\tau\\
&=:\int_t^s \Big[A|\mathbb X(\tau)|+BW_1(\mu_\tau,\nu_\tau)\Big]\,d\tau,
\end{aligned}
\end{align}
where $A>1$ and $B$ do not depend on $\tau$. Now, we combine \eqref{C-3} and \eqref{C-5} to get
\[\mathbb Z(s)\le \int_t^s\Big[A\mathbb Z(\tau)+BW_1(\mu_\tau,\nu_\tau)\Big]\,d\tau,\]
where $\mathbb Z(t):=|\mathbb X(t)|+|\mathbb V(t)|.$
 Finally, we use Gr\"onwall inequality to conclude
\[\mathbb Z(s)\le \int_t^s e^{A(s-\tau)}BW_1(\mu_\tau,\nu_\tau)\,d\tau=C_T\int_t^sW_1(\mu_\tau,\nu_\tau)\,d\tau.\]
\end{proof}

\begin{proposition}\label{P3.13}
Let $\mu_t,\nu_t\in L^\infty([0,T);\mathcal{M}(\bbr^{2d}))$ be two measure-valued solutions to \eqref{A-2} with compactly supported initial condition $\mu_0$ and $\nu_0$.
Moreover, suppose that $\mu_t$ and $\nu_t$ have uniformly bounded moments:
for $0\le t<T$,
\begin{align*}
\int_{\bbr^{2d}}\mu_t(dx,dv)\le 1,\quad \int_{\bbr^{2d}}|v|^2\mu_t(dx,dv)\le m_2,\\
\int_{\bbr^{2d}}\nu_t(dx,dv)\le 1,\quad \int_{\bbr^{2d}}|v|^2\nu_t(dx,dv)\le m_2.\end{align*}
Then, for $0\le t<T$,
\[W_1(\mu_t,\nu_t)\le C_TW_1(\mu_0,\nu_0).\]
\end{proposition}

\begin{proof}
Let $g\in C_0(\bbr^{2d})$ be a test function with $\|g\|_\infty \le 1$ and $\|g\|_{\textup{Lip}}\le 1$.
Then, we obtain, owing to Lemma\,\ref{L3.6}, Remark\,\ref{R3.2} and Lemma\,\ref{lem3.10}, that
\begin{align*}
&\Bigg|\int_{\bbr^{2d}}g(x,v)\big(\mu_t(dx,dv)-\nu_t(dx,dv)\big)\Bigg|\\
&\qquad\le\int_{\bbr^{2d}}\big|g\big(X_\mu(t;0,x,v),V_\mu(t;0,x,v)\big)
{-}g\big(X_\nu(t;0,x,v),V_\nu(t;0,x,v)\big)\big|\mu_0(dx,dv)\\
&\qquad+\left|\int_{\bbr^{2d}}g\big(X_\nu(t;0,x,v),V_\nu(t;0,x,v)\big)
\big(\mu_0(dx,dv){-}\nu_0(dx,dv)\big)\right|\\
&\qquad\le \int_{\bbr^{2d}}\big(|X_\mu(t;0,x,v){-}X_\nu(t;0,x,v)|+|V_\mu(t;0,x,v){-}V_\mu(t;0,x,v)|\big)\mu_0(dx,dv)\\
&\qquad
+W_1(\mu_0,\nu_0)\\
&\qquad\le C_T\int_{\bbr^{2d}}\int_0^tW_1(\mu_\tau,\nu_\tau)d\tau\mu_0(dx,dv)+W_1(\mu_0,\nu_0)\\
&\qquad=C_T\int_0^tW_1(\mu_\tau,\nu_\tau)d\tau+W_1(\mu_0,\nu_0).
\end{align*}
We take a supremum over the function space $\Omega$ to obtain
\[W_1(\mu_t,\nu_t)\le C_T\int_0^tW_1(\mu_\tau,\nu_\tau)\,d\tau+W_1(\mu_0,\nu_0),\]
and after using Gr\"onwall inequality, we get the following uniform stability up to any finite time $T$:
\[W_1(\mu_t,\nu_t)\le C_TW_1(\mu_0,\nu_0).\]
\end{proof}

\subsection{Existence and uniqueness}
In this subsection, we provide the existence and the uniqueness of measure-valued solutions to \eqref{A-2}.
Since the uniqueness follows directly from the stability estimate given in Proposition \ref{P3.13}, we focus on  the existence part.

\noindent{\bf{Proof of Theorem \ref{T.1.1}}}
The idea of proof is taking mean-field limit of particle solutions which can be interpreted as empirical measures. For readability, we separate the proof into three steps.\\

\noindent~$\bullet$ Step A (Approximating initial data $\mu_0$): For a compactly supported Radon measure $\mu_0$, it is well known, for example in \cite{V}, that there exists a sequence of empirical measure $\mu_0^n$ such that
\[\mu_0^n:=\frac{1}{n}\sum_{i=1}^n \delta_{(x_i,v_i)},\quad \lim_{n\to\infty} W_1(\mu_0^n,\mu_0)=0.\]
Therefore, for arbitrary positive constant $\varepsilon>0$, we can find $N=N(\varepsilon)$ such that
\[W_1(\mu_0^n,\mu_0^m)\le \varepsilon ,\quad n,m>N(\varepsilon).\]

\noindent~$\bullet$ Step B (Estimate of $W_1(\mu_t^n,\mu_t^m)$) We denote
\[\mu_0^n:=\frac{1}{n}\sum_{i=1}^n \delta_{(x^0_i,v^0_i)},\quad \mu_0^m:=\frac{1}{m}\sum_{j=1}^m\delta_{(\bar{x}^0_j,\bar{v}^0_j)},\]
and let $(x_i,v_i)$ and $(\bar{x}_i,\bar{v}_i)$ be solutions of the particle system \eqref{A-1} subjected to initial data $(x_i^0,v_i^0)$ and $(\bar{x}_i^0,\bar{v}_i^0)$ respectively. Then,
\[\mu_t^n:=\frac{1}{n}\sum_{i=1}^n \delta_{(x_i,v_i)},\quad \mu_t^m:=\frac{1}{m}\sum_{j=1}^m\delta_{(\bar{x}_j,\bar{v}_j)},\]
are two measure-valued solutions of \eqref{A-2}. Then, we have from the Proposition \ref{P3.13} that
\[W_1(\mu_t^n,\nu_t^m)\le C_TW_1(\mu_0^n,\nu_0^m)\le C_T\varepsilon,\]
which implies $\{\mu_t^n\}_{n\ge1}$ is a Cauchy sequence and converges to $\mu_t$ as $n\to \infty$.\\

\noindent~$\bullet$ Step C (Passing limit)
In view of Definition \ref{D3.1}, we first check the weak continuity of $\mu_t$. For any $g \in C_0^1(\bbr^{2d})$, we have from Lemma \ref{L3.6} that
	\begin{align*}
	\left|\<\mu^n_{t+\Delta t},g\>-\<\mu^n_t,g\>\right|&=\left| \int_{\bbr^{2d}} g(X_\mu^n(t+\Delta t;t,x,v),V_\mu^n(t+\Delta t;t,x,v))-g(x,v)\mu_t(dx,dv)\right| \cr
	&\leq \|g\|_{C^1(\bbr^{2d})}
	\big( |X_\mu^n(t+\Delta t;t,x,v)-X_\mu^n(t)| + |V_\mu^n(t+\Delta t;t,x,v)-V_\mu^n(t)| \big).
	\end{align*}
	This, together with Lemma \ref{L3.5}, implies that there exists a constant $C_T$ such that
	\[
	|X_\mu^n(t+\Delta t;t,x,v)-X_\mu^n(t)| + |V_\mu^n(t+\Delta t;t,x,v)-V_\mu^n(t)| \leq C_T \Delta t.
	\]
	Therefore, the weak continuity holds uniformly for any $n$, which gives that its limit $\mu_t$ is also weakly continuous. Secondly, we need to show that
	\begin{equation}\label{C-6}
	\<\mu_t,g(\cdot,\cdot,t)\>-\<\mu_0,g(\cdot,\cdot,0)\>=\int_0^t\<\mu_s,\partial_sg+v\cdot \nabla_xg+Q\cdot \nabla_v g\>ds,\quad g\in C_0^1(\bbr^{2d}\times[0,T)).
	\end{equation}
Since we already know that the approximated empirical measure $\mu_t^n$ satisfies
\begin{equation}\label{C-7}
\<\mu^n_t,g(\cdot,\cdot,t)\>-\<\mu^n_0,g(\cdot,\cdot,0)\>=\int_0^t\<\mu^n_s,\partial_sg+v\cdot \nabla_xg+Q\cdot \nabla_v g\>ds,
\end{equation}
and hence, we only need to show that each term in \eqref{C-7} converges to \eqref{C-6}. On the other hand, since $\mu_t^n$ converges to $\mu_t$ in the Wasserstein-1 metric, which implies the weak*-convergence of measure, the first two terms trivially converge to their counterparts in \eqref{C-6}
\[\<\mu^n_t,g(\cdot,\cdot,t)\>\to \<\mu_t,g(\cdot,\cdot,t)\>,\quad \<\mu^n_0,g(\cdot,\cdot,0)\>\to\<\mu_0,g(\cdot,\cdot,0)\>.\]
Therefore, it remains to show that the right-hand side of \eqref{C-7} converges to that of \eqref{C-6}:
\[\<\mu^n_s,\partial_sg+v\cdot \nabla_xg +Q(\mu_s^n)\cdot \nabla_vg\>\to \<\mu_s,\partial_sg+v\cdot \nabla_xg +Q(\mu_s)\cdot \nabla_vg\>,\quad \mbox{as} \quad n\to \infty.\]
Again, by the weak*-convergence of measure, it suffices to show
\[\<\mu^n_s,Q(\mu_s^n)\cdot \nabla_vg\>\to \<\mu_s,Q(\mu_s)\cdot \nabla_vg\>.\]
Note that
\begin{align*}
\Big|\<\mu^n_s,&Q(\mu_s^n)\cdot \nabla_vg\>-\<\mu_s,Q(\mu_s)\cdot \nabla_vg\>\Big|=\left|\int_{\bbr^{2d}}Q(\mu_s^n)\cdot \nabla_v g \,d\mu^n_s-\int_{\bbr^{2d}}Q(\mu_s)\cdot \nabla_v g \,d\mu_s\right|\\
&\le \int_{\bbr^{2d}}|Q(\mu_s^n)-Q(\mu_s)| |\nabla_v g| \,d\mu_s^n+\left|\int_{\bbr^{2d}}Q(\mu_s)\cdot \nabla_v g \,d\mu^n_s-\int_{\bbr^{2d}}Q(\mu_s)\cdot \nabla_v g \,d\mu_s\right|\\
&=:\mathcal{I}_{31}+\mathcal{I}_{32}.
\end{align*}

\noindent$\diamond$ (Estimate of $\mathcal{I}_{31}$) : We note that
\begin{equation}\label{C-8}
Q(x,v,\mu_s)=a(\mu_s)+b(x,\mu_s)-(1+c(x,\mu_s))x-c(x,\mu_s)v,
\end{equation}
and we use Lemma \ref{L3.5} and Lemma \ref{L3.8} to estimate
\begin{align*}
|Q(x,v,\mu_s^n)-Q(x,v,\mu_s)|&\le |a(\mu_s^n)-a(\mu_s)|+|b(x,\mu_s^n)-b(x,\mu_s)|\\
&+|c(x,\mu^n_s)-c(x,\mu_s)||x|+|c(x,\mu_s^n)-c(x,\mu_s)||v|\\
&\le C_TW_1(\mu_t^n,\mu_t).
\end{align*}
Therefore, we have
\[\mathcal{I}_{31}\le C_TW_1(\mu_t^n,\mu_t)\|\nabla_vg\|_\infty\to 0,\quad \mbox{as $n\to \infty$}.\]

\noindent$\diamond$ (Estimate of $\mathcal{I}_{32}$) : Thanks to the Remark \ref{R3.2} (2), it suffices to show that
\[\|Q(x,v,\mu_s)\nabla_vg\|_\infty<C_T,\quad \|Q(x,v,\mu_s)\nabla_vg\|_{\textup{Lip}}<C_T.\]
However, the first estimate follows directly from the equation \eqref{C-8} and Lemma \ref{L3.4}:
\[\|Q(x,v,\mu_s)\cdot \nabla_vg\|_\infty\le \|Q(x,v,\mu_s)\|_\infty \|\nabla_vg\|_\infty <C_T.\]
For the second estimate, we observe for any two phase points $(x,v)$ and $ (x^*,v^*)$ lying in the support of $\mu_s$ that
\begin{align*}
|Q(x,v,\mu_s)&\nabla_vg(x,v)-Q(x^*,v^*,\mu_s)\nabla_vg(x^*,v^*)|\\
&\le |Q(x,v,\mu_s)-Q(x^*,v^*,\mu_s)||\nabla_vg(x,v)|+|Q(x^*,v^*,\mu_s)||\nabla_vg(x,v)-\nabla_vg(x^*,v^*)|.
\end{align*}
However, we can show by a similar calculation as in the proof of Lemma \ref{L3.9} that
\begin{align*}
|Q(&x,v,\mu_s)-Q(x^*,v^*,\mu_s)|\\
&\le |b(x,\mu_s)-b(x^*,\mu_s)|+|(1+c(x,\mu_s))x-(1+c(x^*,\mu_s))x^*|+|c(x,\mu_s)v-c(x^*,\mu_s)v^*|\\
&\le|b(x,\mu_s)-b(x^*,\mu_s)|+|x-x^*|+|c(x,\mu_s)-c(x^*,\mu_s)| (|x| + |v|)+ c(x^*,\mu_s)(|x-x^*|+|v-v^*|)\\
&\le C_T(|x-x^*|+|v-v^*|).
\end{align*}
Therefore,
\begin{align*}
|Q(x,v,\mu_s)\nabla_vg(x,v)&-Q(x^*,v^*,\mu_s)\nabla_vg(x^*,v^*)|\\
&\le \Big(C_T\|\nabla_vg\|_\infty +\|Q\|_\infty \|\nabla^2g\|_\infty \Big)(|x-x^*|+|v-v^*|),
\end{align*}
which implies $\|Q(x,v,\mu_s)\nabla_vg\|_{\textup{Lip}}<C_T$. By combining arguments from Step A to Step C, we conclude the proof of Theorem \ref{T.1.1}.

\section{Asymptotic herding behavior of kinetic herding model}
In this section, we study the asymptotic behavior of the measure-valued solutions for (\ref{A-2}). The herding behavior of kinetic model is described in two ways. 
We first define various functionals on the measure-valued solutions.
\subsection{Definitions and basic estimates} We begin by introducing the following energy-like functionals
\begin{equation*}
X(t):=\int_{\bbr^{2d}}|x-x_c|^2\,\mu_t(dx,dv),\quad V(t):=\int_{\bbr^{2d}}|v-v_c|^2\,\mu_t(dx,dv),
\end{equation*}
where the mean variables $x_c$ and $v_c$ are given by
\[x_c(t):=\int_{\bbr^{2d}}x\,\mu(dx,dv),\quad v_c(t)=\int_{\bbr^{2d}}v\,\mu(dx,dv).\]
Note from Remark \ref{R3.1} that we have $v_c(t)=v_c(0)=m_1$ and hence, $x_c(t)=x_c(0)+v_c(0)t$. Therefore, due to the Galilean invariance principle of the model, we may assume $x_c=v_c=0$ without loss of generality. We also need following functionals:
\begin{definition}  \label{aux functionals} We introduce auxiliary functionals defined as follows:

\begin{enumerate}
\item $L^2$-covariance functional:
\[C(X,V)(t):=\int_{\bbr^{2d}}(x-x_c)\cdot (v-v_c)\,\mu_t(dx,dv).\]
\item Auxiliary functional:
\[M(t):=\int_{\bbr^{4d}}\Phi(|x-x^*|)\,\mu_t(dx,dv)\,\mu_t(dx^*,dv^*),\]
where $\Phi$ is defined as
\[\Phi(r):=\frac{1}{2}\int_0^{r^2} \widetilde\phi(s)\,ds,\quad \widetilde\phi(r^2)=\phi(r).\]

\item Weighted energy functional:
\[X_\phi(t):=\int_{\bbr^{4d}}\phi(|x-x^*|)|x-x^*|^2\,\mu_t(dx,dv)\,\mu_t(dx^*,dv^*),\]
\[V_\phi(t):=\int_{\bbr^{4d}}\phi(|x-x^*|)|v-v^*|^2\,\mu_t(dx,dv)\,\mu_t(dx^*,dv^*).\]
\item Weighted covariance functional:
\[C_\phi(X,V)(t):=\int_{\bbr^{4d}}\phi(|x-x^*|)(x-x^*)\cdot (v-v^*)\,\mu_t(dx,dv)\,\mu_t(dx^*,dv^*).\]
\end{enumerate}
\end{definition}

We remark that energy-like functionals $X(t),V(t)$ and $C(X,V)(t)$ can be understood as variances and co-variance of $x$ and $v$ at time $t$, respectively.

In the following lemma, we record how the time derivatives of the above functions are expressed.
\begin{lemma}\label{L3.12}
Let $\mu_t$ be the measure-valued solution of kinetic herding equation \eqref{A-2}. Then,
\begin{align*}
(i)&~\frac{dX}{dt}=2C(X,V),\quad\quad (ii)~\frac{dV}{dt}=-2\lambda_wC(X,V)-\lambda_xC_\phi(X,V)-\lambda_vV_\phi,\\
(iii)&~\frac{dM}{dt}=C_\phi(X,V),\hspace{0.6cm} (iv)~\frac{dC(X,V)}{dt}=V -\lambda_w X- \frac{\lambda_x}{2}X_{\phi}- \frac{\lambda_v}{2} C_{\phi}(X,V).
\end{align*}
\end{lemma}
\begin{proof}
\noindent $(i)$ Choose $g(x,v)=\left|x\right|^2$ in \eqref{A-2} to get
\[\frac{d}{dt}\int_{\bbr^{2d}} \left|x\right|^2\,\mu_t(dx,dv) = 2 \int_{\bbr^{2d}}x\cdot v \,\mu_t(dx,dv)=2C(X,V)(t).\]
\noindent $(ii)$ Similarly, we choose $g(x,v)=|v|^2$ in \eqref{A-2} to obtain
\begin{align*}
\frac{d}{dt}\int_{\bbr^{2d}} |v|^2\,\mu_t(dx,dv)
&=\int_{\bbr^{2d}}2v\cdot Q(x,v,\mu_t)\,\mu_t(dx,dv)\\
&=\int_{\bbr^{2d}} 2v\cdot
\bigg[\int_{\bbr^{2d}}[\lambda_w(x^*-x)+\lambda_x\phi(|x-x^*|)(x^*-x)\cr
&\hspace{2.5cm}+\lambda_v\phi(|x-x^*|)(v^*-v)]\,\mu_t(dx^*,dv^*)\bigg]\,\mu_t(dx,dv)\\
&=:\mathcal{I}_{41} + \mathcal{I}_{42} + \mathcal{I}_{43}.
\end{align*}
\noindent$\bullet$ (Estimate of $\mathcal{I}_{41}$) : Since we assume the zero-mean position condition
\[\int_{\bbr^{2d}}x^*\,\mu_t(dx^*,dv^*)=0,\]
we get
\[\mathcal{I}_{41} = -2\lambda_w\int_{\bbr^{2d}}x\cdot v \,\mu_t(dx,dv)=-2\lambda_wC(X,V)(t).\]
\noindent$\bullet$ (Estimate of $\mathcal{I}_{42}$ and $\mathcal{I}_{43}$): From the definition of $Q_2$ and $Q_3$, we obtain
\begin{align*}
\mathcal{I}_{42} &= 2\lambda_x\int_{\bbr^{2d}}\left(  \int_{\mathbb{R}^{2d}}\phi(|x-x^*|)(x^*-x) \,\mu_t(dx^*,dv^*)\right) \cdot v\,\mu_t(dx,dv)\\
&= -\lambda_x\int_{\bbr^{4d} }\phi(|x-x^*|)(x-x^*)\cdot(v-v^*) \,\mu_t(dx^*,dv^*)\,\mu_t(dx,dv)\\
&=-\lambda_xC_\phi(X,V)	
\end{align*}
and
\begin{align*}
\mathcal{I}_{43}&= 2\lambda_v\int_{\bbr^{2d} }\left(\int_{\mathbb{R}^{2d}}\phi(|x-x^*|)(v^*-v) \,\mu_t(dx^*,dv^*) \right) \cdot v\,\mu_t(dx,dv) \\
&= -\lambda_v\int_{\bbr^{4d}}\phi(|x-x^*|)|v-v^*|^2 \,\mu_t(dx^*,dv^*)\,\mu_t(dx,dv)=-\lambda_vV_\phi.
\end{align*}
We combine all the estimates of $\mathcal{I}_{4i}$, $i=1,2,3$, to obtain the desired estimate.\\

\noindent $(iii)$ It directly comes from the definition of the measure-valued solution that
\begin{align*}
\frac{dM}{dt}&=\int_{\bbr^{4d}}(v\cdot\nabla_x+v^*\cdot\nabla_{x^*})(\Phi(|x-x^*|))\,\mu_t(dx,dv)\,\mu_t(dx^*,dv^*)\\
&=\int_{\bbr^{4d}}\phi(|x-x^*|) ((x-x^*)\cdot v -(x-x^*)\cdot v^*)\,\mu_t(dx,dv)\,\mu_t(dx^*,dv^*)\\
&=\int_{\bbr^{4d}}\phi(|x-x^*|) (x-x^*)\cdot (v-v^*)\,\mu_t(dx,dv)\,\mu_t(dx^*,dv^*)\\
&=C_\phi(X,V).
\end{align*}

\noindent $(iv)$ Again, a direct calculation yields,

\begin{align*}
&\frac{dC(X,V)}{dt}=\int_{\bbr^{2d}}|v|^2\mu_t(dx,dv)+\int_{\bbr^{2d}}x\cdot Q(x,v,\mu_t)\mu_t(dx,dv)\\
&=V+\int_{\bbr^{2d}}x\cdot\left[\int_{\bbr^{2d}}\lambda_w(x^*-x)+\lambda_x\phi(|x-x^*|)(x^*-x)+\lambda_v\phi(|x-x^*|)(v^*-v)\right]\mu_t(dx,dv)\\
&=V-\lambda_wX-\frac{\lambda_x}{2}X_\phi-\frac{\lambda_v}{2}C_\phi(X,V).
\end{align*}
\end{proof}

\subsection{Proof of the Theorem \ref{T.1.2}} We now prove our first result on the asymptotic behavior.
For this,  we introduce a herding energy functional:
\begin{definition}\label{Total E}
	We define our herding energy functional as follows:
	\begin{align*}
	E(t):=\lambda_wX(t)+V(t)+\lambda_xM(t).
	\end{align*}
\end{definition}

\begin{remark}
In the previous section, we have normalized all parameters to be unity because the specific values of parameters were irrelevant in the existence proof. From now on, however, the parameters must satisfy a specific condition (see \eqref{C-12}) to guarantee the emergence of exponential herding behavior. That's why we explicitly revealed the dependence of energy function on the parameters in Definition  \ref{Total E}. We also note that the herding energy $E$ can be understood as a functional measuring the variances of $x$ and $v$ together with the potential energy of the market, in that the  auxiliary functional $M$ can be understood as a potential energy.
\end{remark}
We first need the following lemma \cite{B}.
\begin{lemma}[\bf Barbalat's Lemma \cite{B}]\label{L4.4}
	If a differentiable function $L(t)$ has a finite limit as $t \rightarrow \infty$ and if $L'(t)$ is uniformly continuous (or $L''(t)$ is bounded), then $L'(t) \rightarrow 0$ as $t \rightarrow \infty$.
\end{lemma}


Now, we provide a proof of Theorem \ref{T.1.2}. For reader's convenience, we split the proof into four steps.\\

\noindent{\bf Step I: } We first note from Lemma \ref{L3.12} $(i)$-$(iii)$ that
\begin{align*}
\frac{d}{dt}E(t)
&=\frac{d}{dt}(\lambda_wX(t)+V(t)+\lambda_xM(t))\cr
&=2\lambda_wC(X,V)+(-2\lambda_wC(X,V)-\lambda_xC_\phi(X,V)-\lambda_vV_\phi)
+\lambda_xC_\phi(X,V)\cr
&=-\lambda_vV_\phi\leq0.
\end{align*}
Therefore, $E(t)$ is a positive non-increasing function, and hence it converges to, say  $E^\infty$ as $t\to \infty$:
\begin{align}\label{E=0}
\lim_{t\to\infty} E(t)=E^\infty.
\end{align}

\noindent{\bf Step II:} In this step, we show that
\begin{align}\label{V=0}
\lim_{t\rightarrow\infty}V(t)=0.
\end{align}
For this, we first observe from the definition of functionals $V$ and $V_\phi$ that
	\begin{align*}
	\begin{aligned}
	V_\phi&=\int_{\bbr^{4d}}\phi(|x-x^*|)|v-v^*|^2\mu_t(dx,dv)\mu_t(dx^*,dv^*)\cr
	&\ge\phi_m \int_{\bbr^{4d}}|v-v^*|^2\mu_t(dx,dv)\mu_t(dx^*,dv^*)\\
	&=2\phi_m\int_{\bbr^{2d}}|v|^2\mu_t(dx,dv)=2\phi_mV.
	\end{aligned}
	\end{align*}
	Therefore,
	\[\frac{dE}{dt}=-\lambda_vV_\phi\le -2\lambda_v\phi_mV,\]
	which implies
\begin{align}\label{phi m need}
	\int_0^t V(s)\,ds \le\frac{1}{2\phi_m}\int_0^t V_\phi(s)\,ds = \frac{E(0)-E(t)}{2\phi_m\lambda_v} \le \frac{E(0)}{2\phi_m\lambda_v}.
\end{align}

Consequently, the integral $\int_0^t V(s)\,ds$, which plays the role of $L$ in the Barbalat's lemma, is a bounded increasing function of $t$, and has a finite limit as $t\rightarrow\infty$.

Hence, in order to show $V(t)$ converges to 0 by using Lemma \ref{L4.4}, it suffice to show that the derivative of $V$ is uniformly bounded. However, in virtue of Cauchy-Schwartz inequality, $V'$ can be estimated as	
\begin{align*}
\begin{split}
	\left|V'(t)\right|&=\left|-2\lambda_wC(X,V)-\lambda_xC_\phi(X,V)-\lambda_vV_\phi\right|\cr
	&\leq \lambda_w (X+V)+2\lambda_x\phi_M (X+V)+2\lambda_v\phi_M V\cr
	&\leq (\lambda_w+2\lambda_x\phi_M)X + ( \lambda_w + 2\lambda_x\phi_M + 2\lambda_v\phi_M ) V\cr
	&\leq \max \left\{1+2\frac{\lambda_x}{\lambda_w}\phi_M,\lambda_w + 2\lambda_x\phi_M + 2\lambda_v\phi_M\right\} (\lambda_w X + V)\cr
	&\leq \max \left\{1+2\frac{\lambda_x}{\lambda_w}\phi_M,\lambda_w + 2\lambda_x\phi_M + 2\lambda_v\phi_M\right\} E(0).
\end{split}
\end{align*}	
	Therefore, Lemma \ref{L4.4} implies that $V$ converges to 0 as $t\to \infty$.  \newline
	
\noindent{\bf Step III}: The goal of this step is to prove that
\begin{align}\label{Einfty=0}
E^{\infty} =0.
\end{align}
Suppose in contrary that $E^\infty>0$, so that we can choose $\varepsilon>0$ small enough to satisfy
	\[
	\sqrt{\varepsilon} \left( \sqrt{\varepsilon} + \frac{\lambda_w+\lambda_x\phi_m}{\lambda_w+\lambda_x\phi_M} \sqrt{\varepsilon} + \frac{\lambda_v\phi_M}{2} \sqrt{\frac{E^\infty+\varepsilon}{\lambda_w+\lambda_x\phi_m}} \right) \leq  \frac{\lambda_w+\lambda_x\phi_m}{2(\lambda_w+\lambda_x\phi_M)}E^\infty,
	\]
	which gives
	\begin{align} \label{eps exi}
	\varepsilon -\frac{\lambda_w+\lambda_x\phi_m}{\lambda_w+\lambda_x\phi_M}(E^\infty-\varepsilon) + \frac{\lambda_v\phi_M}{2} \sqrt{\frac{E^\infty+\varepsilon}{\lambda_w+\lambda_x\phi_m}\varepsilon}\leq -\frac{\lambda_w+\lambda_x\phi_m}{2(\lambda_w+ \lambda_x \phi_M)}E^\infty.
	\end{align}
Now, we observe from the result of Step I and Step II and the definition of $E(t)$ that
\[
\lim_{t\to \infty}\big( \lambda_w X(t)+\lambda_x M(t) \big) = \lim_{t\to\infty}\big( E(t)-V(t) \big)=E^\infty.
\]
Also, we observe that
	\begin{align*}
	|X'(t)|&=|2C(X,V)(t)| \cr&\leq 2  \left(\int_{\bbr^{2d}}|x|^2\,\mu_t(dx,dv)\right)^{1/2}\left(\int_{\bbr^{2d}}|v|^2\,\mu_t(dx,dv)\right)^{1/2} \cr &\leq C\sqrt{V}\rightarrow 0.
	\end{align*}
Therefore,  we can find  $T>0$ such that, for  $t\geq T$, we have the following three inequalities:
	\begin{align}\label{xp}
	|X'(t)|=|2C(X,V)(t)|<\varepsilon,
	\end{align}
\begin{align}\label{xp2}
|\lambda_w X(t)+\lambda_x M(t)-E^{\infty}|<\varepsilon,
\end{align}
\begin{align}\label{xp3}
|V(t)|<\varepsilon.
\end{align}
On the other hand, 
since
\begin{align*}
\frac{1}{2}\phi_m|x-x^*|^2\le\Phi(|x-x^*|)=\frac{1}{2}\int_0^{|x-x^*|^2}\phi(\sqrt{s})\,ds\le \frac{1}{2}\phi_M|x-x^*|^2,
\end{align*}
we get 
\begin{align*}
(\lambda_w+ \lambda_x \phi_m)X \leq \lambda_w X+ \lambda_x M \leq (\lambda_w+\lambda_x \phi_M)X.
\end{align*}
Therefore,  we can  estimate $C(X,V)'$ as
\begin{align*}
\begin{split}
		C(X,V)'(t) &= V(t)-\lambda_wX(t) -\frac{\lambda_x}{2}X_{\phi}(t) -\frac{\lambda_v}{2}C_\phi(X,V)(t) \\
		&\leq V(t)-\left(\lambda_w+\lambda_x\phi_m\right)X(t) +\frac{\lambda_v\phi_M}{2}\sqrt{XV}  \cr
		&< \varepsilon -\frac{\lambda_w+\lambda_x\phi_m}{\lambda_w+\lambda_x \phi_M}(\lambda_w X+ \lambda_x M) + \frac{\lambda_v\phi_M}{2} \sqrt{\frac{\lambda_w X+ \lambda_x M}{\lambda_w+\lambda_x\phi_m}}\sqrt{V(t)}.
\end{split}
\end{align*}
Consequently, if $t>T$, we can apply (\ref{xp2}) and (\ref{xp3}) to derive
\begin{align}\label{c ineq}
\begin{split}
C(X,V)'(t)&< \varepsilon -\frac{\lambda_w+\lambda_x\phi_m}{\lambda_w+\lambda_x \phi_M}(E^\infty-\varepsilon) + \frac{\lambda_v\phi_M}{2} \sqrt{\frac{E^\infty+\varepsilon}{\lambda_w+\lambda_x\phi_m}\varepsilon}\cr
&\leq -\frac{\lambda_w+\lambda_x\phi_m}{2(\lambda_w+\lambda_x \phi_M)}E^\infty,
\end{split}
\end{align}
where the last line follows from \eqref{eps exi}.\newline
Now, consider an open interval $(t_1,t_2)$ in $[T,\infty)$ such that
	\[
	(t_2-t_1) > \frac{3(\lambda_w+\lambda_x\phi_M)}{(\lambda_w+\lambda_x\phi_m)E^\infty}\varepsilon, \quad  T<t_1<t_2<\infty.
	\]
	Then, we can deduce from \eqref{c ineq} that
	\begin{align*}
		X'(t_2)-X'(t_1) & = 2\int_{t_1}^{t_2} C'(X,V)(u) du \leq -\int_{t_1}^{t_2}\frac{\lambda_w+\lambda_x\phi_m}{\lambda_w+\lambda_x\phi_M}E^\infty du \\
		&= -(t_2-t_1)\frac{\lambda_w+\lambda_x\phi_m}{\lambda_w+\lambda_x\phi_M}E^\infty <-3\varepsilon.
	\end{align*}
That is,
\begin{align*}
		|X'(t_2)-X'(t_1)| >3\varepsilon.
	\end{align*}
This, however, contradicts to
	\begin{align*}
		|X'(t_2)-X'(t_1)|=|2C(X,V)(t_2)-2C(X,V)(t_1)| < 2\varepsilon,
	\end{align*}
	which follows from \eqref{xp}. 	Therefore, $E^\infty$ must be zero.\newline
	
\noindent{\bf Step IV}  We've shown in Step II that $V(t)$ vanishes as $t\rightarrow\infty$. The  decay of $X(t)$ is obtained from the combined use of
(\ref{E=0}), (\ref{V=0}) and (\ref{Einfty=0}):
	\[
	\lim_{t\to \infty}\big( \lambda_w X(t)+\lambda_x M(t) \big) = \lim_{t\to\infty}\big( E(t)-V(t) \big)=E^\infty=0.
	\]
This completes the proof.

\subsection{Proof of Theorem \ref{T.1.3}}
In this subsection, we provide a detailed proof of Theorem \ref{T.1.3}, which states the exponential decay of variation. In order to obtain exponential decay estimate, we construct a new energy functional $K$ by the linear combination of $X,V$ and $C(X,V)$. We can understand this energy as a measurement of the joint variability of two variables $x$ and $v$.
	\begin{definition} \label{D4.5}
		We define the fast decaying energy functional by
		\begin{align*}
		K(t):= \lambda_w X + \alpha C(X,V) + V,
		\end{align*}
		where $\alpha$ and $\theta$ are given as
		\[\alpha=\frac{2\lambda_x}{\lambda_v \left(1+2\frac{\theta \lambda_w}{\phi_M \lambda_x}\right)},\quad \theta=\frac{1}{2}\min(1,\phi_M).\]
	\end{definition}
In the following, we verify that $K$ decays exponentially fast.
\begin{proposition}\label{decay K}
Let $K$ be a functional defined in Definition \ref{D4.5} with the same choice of $\alpha$ and $\theta$. Then, $K$ is positive and exponentially decays to 0: there exists a positive constant $\beta$ such that
\[
K(t)\le K(0)e^{-\beta t}.
\]
\end{proposition}
\begin{proof}
	From Lemma \ref{L3.12}, we have
	\begin{align}\label{K1' 1}
	\begin{split}
	\frac{d}{dt}K(t) &= 2 \lambda_w C(X,V) +\alpha\bigg\{V -\lambda_w X- \frac{\lambda_x}{2}X_{\phi}- \frac{\lambda_v}{2} C_{\phi}(X,V)\bigg\} \\
	&+ \bigg\{-2\lambda_w C(X,V)
	- \lambda_x C_{\phi}(X,V) - \lambda_vV_{\phi}\bigg\}    \cr
	&= \alpha V-\alpha\lambda_w X
	-\frac{\alpha\lambda_x}{2}\bigg\{  X_{\phi} + \left(\frac{\lambda_v}{\lambda_x} + \frac{2}{\alpha}\right) C_{\phi}(X,V) + \frac{\lambda_v}{\lambda_x}\frac{2}{\alpha}V_{\phi}\bigg\}.
	\end{split}
	\end{align}
	We use $\theta=\frac{1}{2}\min\{1,\phi_M\}$ and $\phi \leq \phi_M$ to obtain
	\begin{align*}
	\begin{split}
	-\alpha\lambda_w X&=-(1-\theta)\alpha\lambda_w X-\theta\alpha\lambda_w X\cr	
	&=-(1-\theta)\alpha\lambda_w X 	-\frac{\theta\alpha\lambda_w}{2\phi_M}X_{\phi}\cr
	&-\frac{\theta\alpha\lambda_w}{2}\int_{\bbr^{4d} } \left(1-\frac{\phi\left(|x^*-x|\right)}{\phi_M}\right)|x-x^*|^2  \,\mu_t(dx^*,dv^*)\,\mu_t(dx,dv)\cr
	&\leq-(1-\theta)\alpha\lambda_w X 	-\frac{\theta\alpha\lambda_w}{2\phi_M}X_{\phi},
	\end{split}
	\end{align*}
	which, combined \eqref{K1' 1}, gives	
\begin{align}\label{K1' 2}
\begin{split}
&\frac{d}{dt}K(t)\cr
&\quad\leq \alpha V-(1-\theta)\alpha\lambda_w X
-\frac{\alpha\lambda_x}{2}\bigg\{  \left(1+ \frac{\theta\lambda_w}{\phi_M\lambda_x}\right)X_{\phi} + \left(\frac{\lambda_v}{\lambda_x} + \frac{2}{\alpha}\right) C_{\phi}(X,V) + \frac{\lambda_v}{\lambda_x}\frac{2}{\alpha}V_{\phi}\bigg\}\cr
&\quad:=\alpha V-(1-\theta)\alpha\lambda_w X+\mathcal{I}_{5}.
	\end{split}
	\end{align}
For brevity, we introduce $L=1+\frac{\theta\lambda_w}{\phi_M\lambda_x}$, $\mu_t=\mu_t(dx,dv)$, $\mu^*_t=\mu_t(dx^*,dv^*) $, $\phi=\phi(|x^*-x|)$,
 and compute $\mathcal{I}_{5}$ as		
	\begin{align*}
	\mathcal{I}_{5}&=-\frac{\alpha\lambda_x}{2} \int_{\bbr^{4d}}\phi \bigg[\bigg(1+ \frac{\theta\lambda_w}{\phi_M\lambda_x}\bigg)|x-x^*|^2 \cr
	&\hspace{3cm}+ \left(\frac{\lambda_v}{\lambda_x} + \frac{2}{\alpha}\right) (x-x^*)\cdot (v-v^*) + \frac{\lambda_v}{\lambda_x}\frac{2}{\alpha}|v-v^*|^2\bigg] \,\mu^*_t\,\mu_t \cr
	&=-\frac{\alpha\lambda_x}{2} \int_{\bbr^{4d}}\phi \bigg[  L|x-x^*|^2 \cr
	&\hspace{3cm}+ \left(\frac{\lambda_v}{\lambda_x} + \frac{2}{\alpha}\right) (x-x^*)\cdot (v-v^*) + \frac{\lambda_v}{\lambda_x}\frac{2}{\alpha}|v-v^*|^2\bigg] \,\mu_t^*\,\mu_t \cr
	&=-\frac{\alpha\lambda_x}{2} \int_{\bbr^{4d}}\phi \left[  L\left(|x-x^*|^2 + \frac{\left(\frac{\lambda_v}{\lambda_x}+\frac{2}{\alpha}\right)}{L} (x-x^*)\cdot (v-v^*) + \frac{\left(\frac{\lambda_v}{\lambda_x}+\frac{2}{\alpha}\right)^2}{4L^2}|v-v^*|^2\right)\right.\\
	&\hspace{3cm} +\bigg(  \left.-\frac{\left(\frac{\lambda_v}{\lambda_x}+\frac{2}{\alpha}\right)^2}{4L}+ \frac{2\lambda_v}{\lambda_x \alpha} \bigg)|v-v^*|^2\right] \,\mu_t^*\,\mu_t.
	\end{align*}
Then, we are able to check
\[
-\frac{\left(\frac{\lambda_v}{\lambda_x}+\frac{2}{\alpha}\right)^2}{4L}+\frac{2\lambda_v}{\alpha\lambda_x}=-\frac{\left(\frac{2L\lambda_v}{\lambda_x}\right)^2}{4L}+\frac{(2L-1)\lambda_v^2}{\lambda_x^2}=(L-1)\left(\frac{\lambda_v}{\lambda_x}\right)^2=\frac{\theta\lambda_w}{\phi_M\lambda_x}\left(\frac{\lambda_v}{\lambda_x}\right)^2>0
\]
for $\alpha=\frac{2\lambda_x}{(2L-1)\lambda_v}$, and
	\begin{align*}
&|x-x^*|^2 + \frac{\left(\frac{\lambda_v}{\lambda_x}+\frac{2}{\alpha}\right)}{L} (x-x^*)\cdot (v-v^*) + \frac{\left(\frac{\lambda_v}{\lambda_x}+\frac{2}{\alpha}\right)^2}{4L^2}|v-v^*|^2\cr
&\hspace{2cm}=\bigg|(x-x^*)+\frac{\left(\frac{\lambda_v}{\lambda_x}+\frac{2}{\alpha}\right)}{2L}(v-v^*)\bigg|^2\cr
&\hspace{2cm}\geq 0
	\end{align*}
	to get
	\begin{align*}
	\mathcal{I}_{5}&\leq
	-\frac{ \theta\lambda_w \lambda_v^2}{ \phi_M\lambda_x^2}\frac{\alpha}{2}   \int_{\bbr^{4d}}\phi(|x^*-x|) |v-v^*|^2 \,\mu_t(dx^*,dv^*)\,\mu_t(dx,dv) \leq
	-\frac{ \theta\lambda_w \lambda_v^2}{ \phi_M\lambda_x^2}\alpha  \phi_m  V.
	\end{align*}
We go back to \eqref{K1' 2} with this estimate to obtain
	\begin{align}\label{K1' 3}
	\begin{split}
	\frac{d}{dt}K(t)&\leq \alpha V-(1-\theta)\alpha\lambda_w X
	-\frac{ \theta\lambda_w \lambda_v^2}{ \phi_M\lambda_x^2}\alpha  \phi_m V\cr
	&\leq -\min\Bigg\{(1-\theta),-1+\frac{ \theta\lambda_w \lambda_v^2}{ \phi_M\lambda_x^2} \phi_m  \Bigg\}\alpha (\lambda_w X + V).
	\end{split}
	\end{align}
	In order to derive a Gr\"onwall type inequality from \eqref{K1' 3}, we choose  $\delta$ to satisfy
\[
0<\delta<1,\quad \frac{\delta^2}{(1-\delta)^2}<\frac{4\lambda_w}{\alpha^2},
\]
	so that
	\begin{align*}
	-(\lambda_w X+V)\leq -\delta(\lambda_w X+\alpha C(X,V)+V).
	\end{align*}
	Now, we combine this with \eqref{K1' 3} to obtain the desired estimate:
	\begin{align}\label{def beta}
	\frac{d}{dt}K(t)
	&\leq -\min\left\{(1-\theta),-1+\frac{ \theta\lambda_w \lambda_v^2}{ \phi_M\lambda_x^2} \phi_m  \right\}\alpha \delta K(t)=:-\beta K(t).
	\end{align}
\end{proof}	
We now prove Theorem \ref{T.1.3}. We start with the decay of $X(t)$. From the choice of $\alpha$, we have
\[
\alpha=\frac{2\lambda_x}{(2L-1)\lambda_v}\le \sqrt{\frac{\theta\lambda_w\phi_m}{\phi_M}}\frac{2}{2L-1}\le \frac{2\sqrt{\lambda_w}}{2L-1}=\frac{2\sqrt{\lambda_w}}{1+2\frac{\theta\lambda_w}{\phi_M\lambda_x}}\le 2\sqrt{\lambda_w}
\]
and hence,
\begin{align}
\begin{aligned} \label{D-12}
K(t)&=\lambda_w X+\alpha C(X,V)+V\cr
&=\int_{\bbr^{2d}} (\lambda_w|x|^2+\alpha x\cdot v+|v|^2)\mu(dx,dv)\\
&=\int_{\bbr^{2d}}\left(\left|\frac{\alpha}{2}x+v\right|^2+\left(\lambda_w-\frac{\alpha^2}{4}\right)|x|^2\right)\mu(dx,dv)\cr
&\geq\left(\lambda_w-\frac{\alpha^2}{4}\right)X(t),
\end{aligned}
\end{align}
which gives the desired estimate for $X(t)$:
\begin{align}\label{X decay}
X(t)\le \frac{K(t)}{\lambda_w-\frac{\alpha^2}{4}}\le \frac{K(0)e^{-\beta t}}{\lambda_w-\frac{\alpha^2}{4}}.
\end{align}
For the decay of $V(t)$, we find
\[
V=K-\lambda_wX-\alpha C(X,V)\le K+\lambda_w X+\alpha|C(X,V)|\le K+\lambda_wX+\alpha\sqrt{XV}.
\]
Therefore, Proposition \ref{decay K}, the decay estimate of $X(t)$ given in (\ref{X decay}) and the boundedness of $V(t)$ implied by \eqref{V=0} give the desired decay estimate for $V(t)$.
This completes the proof.

\section{Existence of classical solution for kinetic equation}
In this section, we consider the existence of classical solutions and their asymptotic herding behavior to the Cauchy problem:
\begin{align}
\begin{aligned}\label{D-1}
&\frac{\partial{f}}{\partial{t}}+v \cdot \nabla_{x}f +\nabla_{v} \cdot \big(Q(f)f\big)=0,\cr
&\hspace{1cm}f(x,v,0)=f_0(x,v).
\end{aligned}
\end{align}

In the following theorem, all the functionals are defined in the exactly same manner as in the previous cases,
with $\mu_t(dx,dv)$ replaced by $f(x,v,t)dxdv$.




\begin{theorem}[Global existence of classical solutions]\label{T4.1} $(1)$ Suppose the communication rate satisfies (\ref{Phi condition}).
Let $f_0\in C^1 (\bbr^d \times \bbr^d)$ be compactly supported.
Then, for any positive time $T>0$, the Cauchy problem \eqref{D-1} has a unique solution $f \in C^1 (\bbr^d \times \bbr^d \times [0,T))$ satisfying
\[
E(t)\leq E(0), \quad t\geq 0.
\]

\noindent$(2)$ Suppose the communication rate also satisfies
\begin{align*}
\quad 0<\phi_m \leq \phi \leq \phi_M,
\end{align*}
for some constants $\phi_m, \phi_M>0$.
Then, we have the following herding phenomena:
\[
\lim_{t\to \infty} (X(t)+V(t))=0.
\]
\noindent $(3)$ We further assume that the parameters satisfy the following conditions:
\[
\frac{ \phi_M\lambda_x^2}{\phi_m\lambda_w \lambda_v^2} <\frac{1}{2}\min(1,\phi_M).
\]
Then, the herding occurs exponentially fast:
\[
X(t)+V(t)\le Ce^{-\frac{\beta}{2}t},
\]
for some positive constants $C$ and $\beta$.
\end{theorem}
Part (1) can be derived by a similar argument as in \cite{HT}. Since the classical solutions are automatically measure-valued solutions, we can inherit
the proof of Theorem \ref{T.1.2} and \ref{T.1.3} to prove (2), (3). We omit the proof for brevity.



{\bf  Acknowledgement}.
Bae was supported by Basic Science Research Program through the National Research Foundation of Korea 
funded by the Ministry
of Education (NRF-2018R1D1A1A09082848).
Yun is supported by Samsung Science and Technology Foundation under Project Number SSTF-BA1801-02.
\bibliographystyle{amsplain}

\end{document}